\documentclass[12pt]{article}%
\usepackage{graphicx}
\usepackage[intlimits]{amsmath}
\usepackage{latexsym}
\usepackage{amsfonts}
\usepackage{amssymb}%
\makeatletter
\newfont{\footsc}{cmcsc10 at 8truept}
\newfont{\footbf}{cmbx10 at 8truept}
\newfont{\footrm}{cmr10 at 10truept}
\newtheorem{theorem}{Theorem}[section]

\newtheorem{claim}[theorem]{Claim}

\newtheorem{fact}[theorem]{Fact}
\newtheorem{lemma}[theorem]{Lemma}

\newtheorem{proposition}[theorem]{Proposition}

\newenvironment{proof}[1][Proof]{\noindent{\textbf {#1}  }}  {\hfill$\Box$\bigskip}

\begin{document}

\title{The number of cliques in graphs of given order and size}
\author{V. Nikiforov\\{\small Department of Mathematical Sciences, University of Memphis, Memphis,
TN 38152}\\{\small e-mail:} {\small vnikifrv@memphis.edu}}
\maketitle

\begin{abstract}
Let $k_{r}\left(  n,m\right)  $ denote the minimum number of $r$-cliques in
graphs with $n$ vertices and $m$ edges. For $r=3,4$ we give a lower bound on
$k_{r}\left(  n,m\right)  $ that approximates $k_{r}\left(  n,m\right)  $ with
an error smaller than $n^{r}/\left(  n^{2}-2m\right)  .$

The solution is based on a constraint minimization of certain multilinear
forms. Our proof combines a combinatorial strategy with extensive analytical
arguments.\bigskip

\textbf{AMS classification: }

\textbf{Keywords: }\textit{number of cliques; mulitilinear forms; Tur\'{a}n
graph.}

\end{abstract}

\section*{Introduction}

Our graph-theoretic notation follows \cite{Bol98}; in particular,\ an
$r$-clique is a complete subgraph on $r$ vertices.

What is the minimum number $k_{r}\left(  n,m\right)  $ of $r$-cliques in
graphs with $n$ vertices and $m$ edges? This problem originated with the
famous graph-theoretical theorem of Tur\'{a}n more than sixty years ago, but
despite numerous attempts, never got a satisfactory solution, see
\cite{Bol76}, \cite{Erd62}, \cite{Erd69}, \cite{Fis92}, \cite{LoSi83}, and
\cite{Raz06} for some highlights of its long history. Most recently, the
problem was discussed in detail in \cite{BCLSV06}.

The best result so far is due to Razborov \cite{Raz06}. Applying tools
developed in \cite{Raz05}, he achieved a remarkable progress for $r=3.$ But
this method failed for $r>3,$ and Razborov challenged the mathematical
community to extend his result.

The aim of this paper is to answer this challenge. We introduce a class of
multilinear forms and find their minima subject to certain constraints. As a
consequence, for $r=3,4$ we obtain a lower bound on $k_{r}\left(  n,m\right)
$, approximating $k_{r}\left(  n,m\right)  $ with an error smaller than
$n^{r}/\left(  n^{2}-2m\right)  .$

In our proof, a combinatorial main strategy cooperates with analytical
arguments using Taylor's expansion, Lagrange's multipliers, compactness,
continuity, and connectedness. We believe that such cooperation can be
developed further and applied to other problems in extremal combinatorics.

It seems likely that these methods will enable the solution of the problem for
$r>4$ as well. With this idea in mind we present all results as general as possible.

\section{Main results}

Suppose $1\leq r\leq n,$ let $\left[  n\right]  =\left\{  1,\ldots,n\right\}
,$ and write $\binom{\left[  n\right]  }{r}$ for the set of $r$-subsets of
$\left[  n\right]  .$ For a symmetric $n\times n$ matrix $A=\left(
a_{ij}\right)  $ and a vector $\mathbf{x}=\left(  x_{1},\ldots,x_{n}\right)
,$ set
\begin{equation}
L_{r}\left(  A,\mathbf{x}\right)  =%
%TCIMACRO{\dsum _{X\in\binom{\left[  n\right]  }{r}}}%
%BeginExpansion
{\displaystyle\sum_{X\in\binom{\left[  n\right]  }{r}}}
%EndExpansion
\text{ }%
%TCIMACRO{\tprod \limits_{i,j\in X,\text{ }i<j}}%
%BeginExpansion
{\textstyle\prod\limits_{i,j\in X,\text{ }i<j}}
%EndExpansion
a_{ij}%
%TCIMACRO{\tprod \limits_{i\in X}}%
%BeginExpansion
{\textstyle\prod\limits_{i\in X}}
%EndExpansion
x_{i}. \label{defl1}%
\end{equation}
Define the set $\mathcal{A}\left(  n\right)  $ of symmetric $n\times n$
matrices $A=\left(  a_{ij}\right)  $ by
\[
\mathcal{A}\left(  n\right)  =\left\{  A:\text{ }a_{ii}=0\text{ and }0\leq
a_{ij}=a_{ji}\leq1\text{ for all }i,j\in\left[  n\right]  \right\}  .
\]
Our main goal is to find $\min L_{r}\left(  A,\mathbf{x}\right)  $ subject to
the constraints%
\[
A\in\mathcal{A}\left(  n\right)  ,\text{ }\mathbf{x}\geq0,\text{ }L_{1}\left(
A,\mathbf{x}\right)  =b,\text{\ and }L_{2}\left(  A,\mathbf{x}\right)  =c,
\]
where $b$ and $c$ are fixed positive numbers. Since every $L_{s}\left(
A,\mathbf{x}\right)  $ is homogenous of first degree in each $x_{i},$ for
simplicity we assume that $b=1$ and study%
\begin{equation}
\min\left\{  L_{r}\left(  A,\mathbf{x}\right)  :\left(  A,\mathbf{x}\right)
\in\mathcal{S}_{n}\left(  c\right)  \right\}  , \label{mainp}%
\end{equation}
where $\mathcal{S}_{n}\left(  c\right)  $ is the set of pairs $\left(
A,\mathbf{x}\right)  $ defined as%
\[
\mathcal{S}_{n}\left(  c\right)  =\{\left(  A,\mathbf{x}\right)  :\text{ }%
A\in\mathcal{A}\left(  n\right)  ,\text{\ }\mathbf{x}\geq0,\text{ }%
L_{1}\left(  A,\mathbf{x}\right)  =1,\text{ and }L_{2}\left(  A,\mathbf{x}%
\right)  =c\}.
\]
Note that $\mathcal{S}_{n}\left(  c\right)  $ is compact since the functions
$L_{s}\left(  A,\mathbf{x}\right)  $ are continuous; hence (\ref{mainp}) is
defined whenever $\mathcal{S}_{n}\left(  c\right)  $ is nonempty. The
following proposition, proved in \ref{pp0}, describes when $\mathcal{S}%
_{n}\left(  c\right)  \neq\varnothing$.

\begin{proposition}
\label{pro0}$\mathcal{S}_{n}\left(  c\right)  $ is nonempty if and only if
$c<1/2$ and $n\geq\left\lceil 1/\left(  1-2c\right)  \right\rceil .$
\end{proposition}

Hereafter we assume that $0<c<1/2$ and set $\xi\left(  c\right)  =\left\lceil
1/\left(  1-2c\right)  \right\rceil .$

To find (\ref{mainp}), we solve a seemingly more general problem: for all
$c\in\left(  0,1/2\right)  ,$ $n\geq\xi\left(  c\right)  ,$ and $3\leq r\leq
n,$ find
\[
\varphi_{r}\left(  n,c\right)  =\min\left\{  L_{r}\left(  A,\mathbf{x}\right)
:\text{ }r\leq k\leq n,\text{ }\left(  A,\mathbf{x}\right)  \in\mathcal{S}%
_{k}\left(  c\right)  \right\}  .
\]
We obtain the solution of (\ref{mainp}) by showing that, in fact, $\varphi
_{r}\left(  n,c\right)  $ is independent of $n$.

To state $\varphi_{r}\left(  n,c\right)  $ precisely, we need some
preparation. Set $s=\xi\left(  c\right)  $ and note that the system%
\begin{align}
\binom{s-1}{2}x^{2}+\left(  s-1\right)  xy  &  =c,\label{cond1}\\
\left(  s-1\right)  x+y  &  =1,\label{cond2}\\
x  &  \geq y\nonumber
\end{align}
has a unique solution
\begin{equation}
x=\frac{1}{s}+\frac{1}{s}\sqrt{1-\frac{2s}{s-1}c},\text{ \ \ \ }y=\frac{1}%
{s}-\frac{s-1}{s}\sqrt{1-\frac{2s}{s-1}c}. \label{sol}%
\end{equation}
Write $\mathbf{x}_{c}$ for the $s$-vector $\left(  x,\ldots,x,y\right)  $ and
let $A_{s}\in\mathcal{A}\left(  s\right)  $ be the matrix with all
off-diagonal entries equal to $1.$ Note that equations (\ref{cond1}) and
(\ref{cond2}) give $\left(  A_{s},\mathbf{x}_{c}\right)  \in\mathcal{S}%
_{s}\left(  c\right)  .$

Setting $\varphi_{r}\left(  c\right)  =L_{r}\left(  A_{s},\mathbf{x}%
_{c}\right)  ,$ we arrive at the main result in this section.

\begin{theorem}
\label{mainTh}If $c\in\left(  0,1/2\right)  ,$ $r\in\left\{  3,4\right\}  ,$
and $r\leq\xi\left(  c\right)  \leq n,$ then $\varphi_{r}\left(  n,c\right)
=\varphi_{r}\left(  c\right)  .$
\end{theorem}

Note first that the premise $r\leq\xi\left(  c\right)  $ is not restrictive,
for, $\varphi_{r}\left(  n,c\right)  =0$ whenever $r>\xi\left(  c\right)  .$
Indeed, assume that $r>\xi\left(  c\right)  $ and write $\mathbf{y}$ for the
$r$-vector $\left(  x,\ldots,x,y,0,\ldots,0\right)  $ whose last $r-s$ entries
are zero. Writing $B$ for the $r\times r$ matrix with $A_{s}$ as a principal
submatrix in the first $s$ rows and with all other entries being zero, we see
that $\left(  B,\mathbf{y}\right)  \in\mathcal{S}_{r}\left(  c\right)  $ and
$L_{r}\left(  B,\mathbf{y}\right)  =0;$ hence $\varphi_{r}\left(  n,c\right)
=0,$ as claimed.

Next, note an explicit form of $\varphi_{r}\left(  c\right)  :$%
\begin{align*}
\varphi_{r}\left(  c\right)   &  =\binom{s-1}{r}x^{r}+\binom{s-1}{r-1}%
x^{r-1}y\\
&  =\binom{s}{r}\frac{1}{s^{r}}\left(  1-\left(  r-1\right)  \sqrt{1-\frac
{2s}{s-1}c}\right)  \left(  1+\sqrt{1-\frac{2s}{s-1}c}\right)  ^{r-1}.
\end{align*}

Since $\varphi_{r}\left(  c\right)  $ is defined via the discontinuous step
function $\xi\left(  c\right)  ,$ the following properties of $\varphi
_{r}\left(  c\right)  $ are worth stating:

- $\varphi_{r}\left(  c\right)  $ is continuous for $c\in\left(  0,1/2\right)
;$

- $\varphi_{r}\left(  c\right)  =0$ for $c\in\left(  0,1/4\right]  $ and is
increasing for $c\in\left(  1/4,1/2\right)  ;$

- $\varphi_{r}\left(  c\right)  $ is differentiable and concave in any
interval $\left(  \left(  s-1\right)  /2s,s/2\left(  s+1\right)  \right)  .$

\subsection{The number of cliques}

Write $k_{r}\left(  G\right)  $ for the number of $r$-cliques of a graph $G$
and let us outline the connection of Theorem \ref{mainTh} to $k_{r}\left(
G\right)  $. Let
\[
k_{r}\left(  n,m\right)  =\min\left\{  k_{r}\left(  G\right)  :\text{ }G\text{
has }n\text{ vertices and }m\text{ edges}\right\}  ,
\]
and suppose that $k_{r}\left(  n,m\right)  $ is attained on a graph $G$ with
adjacency matrix $A=\left(  a_{ij}\right)  .$ Clearly, for every $X\in
\binom{\left[  n\right]  }{r},$%
\[%
%TCIMACRO{\tprod \limits_{i,j\in X,\text{ }i<j}}%
%BeginExpansion
{\textstyle\prod\limits_{i,j\in X,\text{ }i<j}}
%EndExpansion
a_{ij}=\left\{
\begin{array}
[c]{ll}%
1, & \text{if }X\text{ induces an }r\text{-clique in }G,\\
0, & \text{otherwise.}%
\end{array}
\right.  .
\]
Hence, letting $\mathbf{x}=\left(  1/n,\ldots,1/n\right)  ,$ we see that
\[
L_{1}\left(  A,\mathbf{x}\right)  =1,\text{ }L_{2}\left(  A,\mathbf{x}\right)
=m/n^{2},\text{\ and }L_{r}\left(  A,\mathbf{x}\right)  =k_{r}\left(
G\right)  /n^{r};
\]
thus Theorem \ref{mainTh} gives
\[
k_{r}\left(  n,m\right)  \geq\varphi_{r}\left(  n,m/n^{2}\right)
n^{r}=\varphi_{r}\left(  m/n^{2}\right)  n^{r}.
\]
Setting $s=\xi\left(  m/n^{2}\right)  =\left\lceil 1/\left(  1-2m/n^{2}%
\right)  \right\rceil ,$ we obtain an explicit form of this inequality%
\begin{equation}
k_{r}\left(  n,m\right)  \geq\binom{s}{r}\frac{1}{s^{r}}\left(  n-\left(
r-1\right)  \sqrt{n^{2}-\frac{2sm}{s-1}}\right)  \left(  n+\sqrt{n^{2}%
-\frac{2sm}{s-1}}\right)  ^{r-1}. \label{minc}%
\end{equation}

Inequality (\ref{minc}) turns out to be rather tight, as stated below and
proved in Section \ref{apx}.

\begin{theorem}
\label{thmeq}
\[
k_{r}\left(  n,m\right)  <\varphi_{r}\left(  \frac{m}{n^{2}}\right)
n^{r}+\frac{n^{r}}{n^{2}-2m}.
\]

\end{theorem}

Note, in particular, that if $m<\left(  1/2-\varepsilon\right)  n^{2},$ then
\[
k_{r}\left(  n,m\right)  <\varphi_{r}\left(  m/n^{2}\right)  n^{r}%
+n^{r-2}/2\varepsilon,
\]
so the order of the error is lower than expected.

\subsubsection*{Known previous results}

For $n^{2}/4\leq m\leq n^{2}/3$ inequality (\ref{minc}) was first proved by
Fisher \cite{Fis92}. He showed that
\[
k_{3}\left(  n,m\right)  \geq\frac{9nm-2n^{3}-2\left(  n^{2}-3m\right)
^{3/2}}{27}=\varphi_{3}\left(  m/n^{2}\right)  n^{3},
\]
but did not discuss how close the two sides of this inequality are.

Recently Razborov \cite{Raz06} showed that for every fixed $c\in\left(
0,1/2\right)  ,$
\[
k_{3}\left(  n,\left\lceil cn^{2}\right\rceil \right)  =\varphi_{3}\left(
c\right)  n^{3}+o\left(  n^{3}\right)  .
\]
Unfortunately, his approach, based on \cite{Raz05}, provides no clues
whatsoever how large the $o\left(  n^{3}\right)  $ term is; in particular, in
his approach this term is not uniformly bounded when $c$ approaches $1/2.$ In
\cite{Raz06} Razborov challenged the mathematical community to prove that
$k_{r}\left(  n,\left\lceil cn^{2}\right\rceil \right)  =\varphi_{r}\left(
c\right)  n^{r}+o\left(  n^{r}\right)  $ for $r>3$. Our Theorem \ref{mainTh}
proves this equality for $r=4.$

\section{Proof of Theorem \ref{mainTh}}

The following simple lemma will be used in the proof of Theorem \ref{mainTh}.

\begin{lemma}
\label{le1}Let $0\leq c\leq a$ and $0\leq d\leq b.$ If $0\leq x\leq\min\left(
a,b\right)  $ and $0\leq y\leq\min\left(  c,d\right)  ,$ then
\[
\left(  a-c\right)  \left(  b-d\right)  +x\left(  c+d\right)  +y\left(
a+b\right)  -\left(  x+y\right)  ^{2}\geq0
\]

\end{lemma}

\begin{proof}
Set $P=x\left(  c+d\right)  +y\left(  a+b\right)  -\left(  x+y\right)  ^{2}.$
Since $\left(  a-c\right)  \left(  b-d\right)  \geq0,$ we may and shall
suppose that $P<0.$ By symmetry, we also suppose that $a\geq b.$ If $x+y\leq
b,$ by $c+d\leq a+b$ we have
\[
P\geq\left(  x+y\right)  \left(  c+d\right)  +y\left(  a+b-c-d\right)
-\left(  x+y\right)  ^{2}\geq\left(  x+y\right)  \left(  c+d\right)  -\left(
x+y\right)  ^{2};
\]
hence, $P<0$ implies that $b>c+d$ and $P\geq b\left(  c+d\right)  -b^{2}$. Now
the proof is completed by
\[
\left(  a-c\right)  \left(  b-d\right)  +b\left(  c+d\right)  -b^{2}=\left(
a-b\right)  \left(  b-d\right)  +cd>0.
\]
If $x+y>b,$ by $c+d\leq a+b,$ we have%
\[
P\geq b\left(  c+d\right)  +y\left(  a+b\right)  -\left(  b+y\right)
^{2}=b\left(  c+d\right)  +y\left(  a-b\right)  -b^{2}-y^{2};
\]
hence, $P<0$ implies that $\min\left(  c,d\right)  >a-b$ and
\[
P\geq b\left(  c+d\right)  +\min\left(  c,d\right)  \left(  a-b\right)
-b^{2}-\left(  \min\left(  c,d\right)  \right)  ^{2}.
\]
If $d\geq c,$ we get%
\begin{align*}
\left(  a-c\right)  \left(  b-d\right)  +P &  \geq\left(  a-c\right)  \left(
b-d\right)  -b\left(  b-d\right)  +c\left(  a-c\right)  \\
&  \geq\left(  a-c\right)  \left(  b-d\right)  -b\left(  b-d\right)  +c\left(
b-d\right)  =\left(  a-b\right)  \left(  b-d\right)  \geq0.
\end{align*}
If $c\geq d,$ we get%
\begin{align*}
\left(  a-c\right)  \left(  b-d\right)  +P  & \geq\left(  a-c\right)  \left(
b-d\right)  +b\left(  c+d\right)  +d\left(  a-b\right)  -b^{2}-d^{2}\\
& =a\left(  a-b\right)  +c\left(  c-d\right)  \geq0,
\end{align*}
completing the proof of Lemma \ref{le1}.
\end{proof}

Next we show that $\varphi_{r}\left(  n,c\right)  $ increases in $c$ whenever
$\varphi_{r}\left(  n,c\right)  >0.$

\begin{proposition}
\label{pro1}Let $c\in\left(  0,1/2\right)  $ and $3\leq r\leq\xi\left(
c\right)  \leq n.$ If $\varphi_{r}\left(  n,c\right)  >0$ and $0<c_{0}<c,$
then $\varphi_{r}\left(  n,c\right)  >\varphi_{r}\left(  n,c_{0}\right)  .$
\end{proposition}

\begin{proof}
Suppose that
\[
\xi\left(  c\right)  \leq k\leq n,\text{\ }\left(  A,\mathbf{x}\right)
\in\mathcal{S}_{k}\left(  c\right)  ,\text{ and \ }\varphi_{r}\left(
n,c\right)  =L_{r}\left(  A,\mathbf{x}\right)  .
\]
Setting $\alpha=c_{0}/c,$ we see that $\alpha A\in\mathcal{A}\left(  k\right)
$ and
\[
L_{2}\left(  \alpha A,\mathbf{x}\right)  =\alpha L_{r}\left(  A,\mathbf{x}%
\right)  =c_{0};
\]
thus $\left(  \alpha A,\mathbf{x}\right)  \in\mathcal{S}_{k}\left(
c_{0}\right)  .$ Hence we obtain
\[
\varphi_{r}\left(  n,c\right)  =L_{r}\left(  A,\mathbf{x}\right)
=\alpha^{-\binom{r}{2}}L_{r}\left(  \alpha A,\mathbf{x}\right)  >L_{r}\left(
\alpha A,\mathbf{x}\right)  \geq\varphi_{r}\left(  n,c_{0}\right)  ,
\]
completing the proof of Proposition \ref{pro1}.
\end{proof}

\subsection*{Proof of Theorem \ref{mainTh}}

Let us first define a set of $n$-vectors $\mathcal{X}\left(  n\right)  $ by%
\[
\mathcal{X}\left(  n\right)  =\left\{  \left(  x_{1},\ldots,x_{n}\right)
:x_{1}+\cdots+x_{n}=1\text{ and }x_{i}\geq0,\text{ }1\leq i\leq n\right\}  .
\]
Now the conditions $\mathbf{x}\in\mathcal{X}\left(  n\right)  $ is equivalent
to $\mathbf{x}\geq0$ and $L_{1}\left(  A,\mathbf{x}\right)  =1.$

Assume for a contradiction that the theorem fails: let
\begin{equation}
c\in\left(  0,1/2\right)  ,\text{ }3\leq r\leq\xi\left(  c\right)  \leq
n,\text{\ }A=\left(  a_{ij}\right)  ,\text{\ }\mathbf{x}=\left(  x_{1}%
,\ldots,x_{n}\right)  ,\text{\ and\ }\left(  A,\mathbf{x}\right)
\in\mathcal{S}_{n}\left(  c\right)  \label{eq6}%
\end{equation}
be such that
\begin{equation}
\varphi_{r}\left(  n,c\right)  =L_{r}\left(  A,\mathbf{x}\right)  <\varphi
_{r}\left(  c\right)  . \label{eq2}%
\end{equation}
Assume that $n$ is the minimum integer with this property for all $c\in\left(
0,1/2\right)  ,$ and that, among all pairs $\left(  A,\mathbf{x}\right)
\in\mathcal{S}_{n}\left(  c\right)  ,$ $A$ has the maximum number of zero
entries. Hereafter we shall refer to this assumption as the \textquotedblleft
main assumption\textquotedblright. The most important consequence of the main
assumption is the following

\begin{claim}
\label{cl0}If $\left(  A,\mathbf{y}\right)  \in\mathcal{S}_{n}\left(
c\right)  $ and $\varphi_{r}\left(  n,c\right)  =L_{r}\left(  A,\mathbf{y}%
\right)  ,$ then $\mathbf{y}$ has no zero entries.$\hfill\square$
\end{claim}

Next we introduce some notation and conventions to simplify the presentation.
For short,\ for every $i,j,\ldots,k\in\left[  n\right]  ,$ set%
\[
C_{i}=\frac{\partial L_{2}\left(  A,\mathbf{x}\right)  }{\partial x_{i}%
},\text{ \ \ }C_{ij}=\frac{\partial L_{2}\left(  A,\mathbf{x}\right)
}{\partial x_{i}\partial x_{j}},\text{ \ \ }D_{ij\ldots k}=\frac{\partial
L_{r}\left(  A,\mathbf{x}\right)  }{\partial x_{i}\partial x_{j}\cdots\partial
x_{k}},
\]
and note that
\begin{equation}
C_{ij}=a_{ij},\text{ \ \ and \ \ }\frac{\partial L_{r}\left(  A,\mathbf{x}%
\right)  }{\partial a_{ij}}a_{ij}=D_{ij}x_{i}x_{j}. \label{deq}%
\end{equation}

Letting $\mathbf{y}=\left(  x_{1}+\Delta_{1},\ldots,x_{n}+\Delta_{n}\right)
,$ Taylor's formula gives
\begin{equation}
L_{2}\left(  A,\mathbf{y}\right)  -L_{2}\left(  A,\mathbf{x}\right)  =%
%TCIMACRO{\dsum \limits_{i=1}^{n}}%
%BeginExpansion
{\displaystyle\sum\limits_{i=1}^{n}}
%EndExpansion
C_{i}\Delta_{i}+\sum\limits_{1\leq i<j\leq n}C_{ij}\Delta_{i}\Delta_{j}
\label{Tay2}%
\end{equation}
and%
\begin{equation}
L_{r}\left(  A,\mathbf{y}\right)  -L_{r}\left(  A,\mathbf{x}\right)  =%
%TCIMACRO{\dsum \limits_{s=1}^{r}}%
%BeginExpansion
{\displaystyle\sum\limits_{s=1}^{r}}
%EndExpansion
\text{ }%
%TCIMACRO{\dsum \limits_{1\leq i_{1}<\cdots<i_{s}\leq n}}%
%BeginExpansion
{\displaystyle\sum\limits_{1\leq i_{1}<\cdots<i_{s}\leq n}}
%EndExpansion
D_{i_{1}\ldots i_{s}}\Delta_{i_{1}}\cdots\text{ }\Delta_{i_{s}}. \label{Tayr}%
\end{equation}

We shall use extensively Lagrange multipliers. Since $\mathbf{x}>0,$ by
Lagrange's method, there exist $\lambda$ and $\mu$ such that
\begin{equation}
D_{i}=\lambda C_{i}+\mu\label{lag1}%
\end{equation}
for all $i\in\left[  n\right]  $. Likewise, if $0<a_{ij}<1,$ we have%
\[
\frac{\partial L_{r}\left(  A,\mathbf{x}\right)  }{\partial a_{ij}}%
=\lambda\frac{\partial L_{2}\left(  A,\mathbf{x}\right)  }{\partial a_{ij}%
}=\lambda x_{i}x_{j},
\]
and so, in view of (\ref{deq}),
\begin{equation}
D_{ij}=\lambda a_{ij}\ \ \text{whenever \ }0<a_{ij}<1. \label{lag2}%
\end{equation}

The rest of the proof is presented in a sequence of formal claims. First we
show that $\varphi_{r}\left(  n,c\right)  $ is attained on a $\left(
0,1\right)  $-matrix $A$.$^{{}}$

\begin{claim}
\label{cl00}Let $\left(  A,\mathbf{x}\right)  \in\mathcal{S}_{n}\left(
c\right)  $ satisfy (\ref{eq6}) and (\ref{eq2}), and suppose that $A$ has the
smallest number of entries $a_{ij}$ such that $0<a_{ij}<1$. Then $A$ is a
$\left(  0,1\right)  $-matrix.
\end{claim}

\begin{proof}
Assume for a contradiction that $i,j\in\left[  n\right]  $ and $0<a_{ij}<1.$
By symmetry we suppose that $C_{i}\geq C_{j}.$ Let
\begin{equation}
f\left(  \alpha\right)  =\frac{a_{ij}\alpha^{2}-\left(  C_{i}-C_{j}\right)
\alpha}{\left(  x_{i}+\alpha\right)  \left(  x_{j}-\alpha\right)  },\label{eq}%
\end{equation}
and suppose that $\alpha$ satisfies%
\begin{equation}
0<\alpha<x_{j}\text{ and \ }0\leq a_{ij}+f\left(  \alpha\right)
\leq1.\label{cond0}%
\end{equation}

Let $\mathbf{y}_{\alpha}=\left(  x_{1}+\Delta_{1},\ldots,x_{n}+\Delta
_{n}\right)  ,$ where
\begin{equation}
\Delta_{i}=\alpha,\text{ \ \ }\Delta_{j}=-\alpha,\text{ \ \ and \ \ }%
\Delta_{l}=0\text{\ for }l\in\left[  n\right]  \backslash\left\{  i,j\right\}
, \label{defy}%
\end{equation}
and define the $n\times n$ matrix $B_{\alpha}=\left(  b_{ij}\right)  $ by%
\begin{equation}
b_{ij}=b_{ji}=a_{ij}+f\left(  \alpha\right)  \text{ \ \ and \ \ }b_{pq}%
=a_{pq}\text{ for }\left\{  p,q\right\}  \neq\left\{  i,j\right\}  .
\label{defB}%
\end{equation}

Note that $B_{\alpha}\in\mathcal{A}\left(  n\right)  ,$ $\mathbf{y}_{\alpha
}\in\mathcal{X}\left(  n\right)  ,$ and%
\[
L_{2}\left(  B_{\alpha},\mathbf{y}_{\alpha}\right)  -L_{2}\left(
A,\mathbf{y}_{\alpha}\right)  =f\left(  \alpha\right)  \frac{\partial
L_{2}\left(  A,\mathbf{y}_{\alpha}\right)  }{\partial a_{ij}}=f\left(
\alpha\right)  \left(  x_{i}+\alpha\right)  \left(  x_{j}-\alpha\right)  .
\]
Hence, Taylor's expansion (\ref{Tay2}) and equation (\ref{eq}) give
\begin{align*}
L_{2}\left(  B_{\alpha},\mathbf{y}_{\alpha}\right)  -L_{2}\left(
A,\mathbf{x}\right)   &  =L_{2}\left(  A,\mathbf{y}_{\alpha}\right)
-L_{2}\left(  A,\mathbf{x}\right)  +f\left(  \alpha\right)  \left(
x_{i}+\alpha\right)  \left(  x_{j}-\alpha\right) \\
&  =\left(  C_{i}-C_{j}\right)  \alpha-a_{ij}\alpha^{2}+f\left(
\alpha\right)  \left(  x_{i}+\alpha\right)  \left(  x_{j}-\alpha\right)  =0;
\end{align*}
thus $\left(  B_{\alpha},\mathbf{y}_{\alpha}\right)  \in\mathcal{S}_{n}\left(
c\right)  .$

Note also that, in view of (\ref{deq}),
\[
L_{r}\left(  B_{\alpha},\mathbf{y}_{\alpha}\right)  -L_{r}\left(
A,\mathbf{y}_{\alpha}\right)  =\frac{\partial L_{r}\left(  A,\mathbf{y}%
_{\alpha}\right)  }{\partial a_{ij}}f\left(  \alpha\right)  =f\left(
\alpha\right)  y_{i}y_{j}\frac{D_{ij}}{a_{ij}}=f\left(  \alpha\right)  \left(
x_{i}+\alpha\right)  \left(  x_{j}-\alpha\right)  \frac{D_{ij}}{a_{ij}}.
\]
Hence Taylor's expansion (\ref{Tayr}), Lagrange's conditions (\ref{lag1}) and
(\ref{lag2}), and equation (\ref{eq}) give%
\begin{align*}
L_{r}\left(  B_{\alpha},\mathbf{y}_{\alpha}\right)  -L_{r}\left(
A,\mathbf{x}\right)   &  =L_{r}\left(  A,\mathbf{y}_{\alpha}\right)
-L_{r}\left(  A,\mathbf{x}\right)  +f\left(  \alpha\right)  \left(
x_{i}+\alpha\right)  \left(  x_{j}-\alpha\right)  \frac{D_{ij}}{a_{ij}}\\
&  =\left(  D_{i}-D_{j}\right)  \alpha-D_{ij}\alpha^{2}+f\left(
\alpha\right)  \left(  x_{i}+\alpha\right)  \left(  x_{j}-\alpha\right)
\frac{D_{ij}}{a_{ij}}\\
&  =\lambda\left(  C_{i}-C_{j}\right)  \alpha-D_{ij}\alpha^{2}+f\left(
\alpha\right)  \left(  x_{i}+\alpha\right)  \left(  x_{j}-\alpha\right)
\frac{D_{ij}}{a_{ij}}\\
&  =\frac{D_{ij}}{a_{ij}}\left(  C_{i}-C_{j}\right)  \alpha-D_{ij}\alpha
^{2}+f\left(  \alpha\right)  \left(  x_{i}+\alpha\right)  \left(  x_{j}%
-\alpha\right)  \frac{D_{ij}}{a_{ij}}\\
&  =\frac{D_{ij}}{a_{ij}}\left(  \left(  C_{i}-C_{j}\right)  \alpha
-a_{ij}\alpha^{2}+a_{ij}\alpha^{2}-\left(  C_{i}-C_{j}\right)  \alpha\right)
=0.
\end{align*}

If there exists $\alpha\in\left(  0,x_{j}\right)  $ such that $a_{ij}+f\left(
\alpha\right)  =0$ or $a_{ij}+f\left(  \alpha\right)  =1,$ we see that the
matrix $B_{\alpha}$ has fewer entries belonging to $\left(  0,1\right)  $ than
$A$, contradicting the hypothesis and completing the proof. Assume therefore
that $0<a_{ij}+f\left(  \alpha\right)  <1$ for all $\alpha\in\left(
0,x_{j}\right)  .$ This condition implies that
\[
a_{ij}x_{j}=C_{i}-C_{j},
\]
for, otherwise $\lim_{\alpha\rightarrow x_{j}}\left\vert f\left(
\alpha\right)  \right\vert =\infty,$ and so, either $a_{ij}+f\left(
\alpha\right)  =0$ or $a_{ij}+f\left(  \alpha\right)  =1$ for some $\alpha
\in\left(  0,x_{j}\right)  $.

Now, extending $f\left(  \alpha\right)  $ continuously for $\alpha=x_{j}$ by
\[
f\left(  x_{j}\right)  =\lim_{\alpha\rightarrow x_{j}}f\left(  \alpha\right)
=\lim_{\alpha\rightarrow x_{j}}\frac{a_{ij}\alpha\left(  \alpha-x_{j}\right)
}{\left(  x_{i}+\alpha\right)  \left(  x_{j}-\alpha\right)  }=-\frac
{a_{ij}x_{j}}{x_{i}+x_{j}},
\]
and defining $\mathbf{y}_{x_{j}}$ by (\ref{defy}) and $B_{x_{j}}$ by
(\ref{defB}), we obtain
\[
L_{r}\left(  B_{x_{j}},\mathbf{y}_{x_{j}}\right)  -\varphi_{r}\left(
n,c\right)  =L_{r}\left(  B_{x_{j}},\mathbf{y}_{x_{j}}\right)  -L_{r}\left(
A,\mathbf{x}\right)  =0.
\]
contradicting Claim \ref{cl0} since the $j$th entry of $\mathbf{y}_{x_{j}}$ is
zero. This completes the proof of Claim \ref{cl00}.
\end{proof}

Since $A$ is a $\left(  0,1\right)  $-matrix with a zero main diagonal, it is
the adjacency matrix of some graph $G$ with vertex set $\left[  n\right]  .$
Write $E\left(  G\right)  $ for the edge set of $G$ and let us restate the
functions $L_{r}\left(  A,\mathbf{x}\right)  $ in terms of $G$. We have
\[
L_{2}\left(  A,\mathbf{x}\right)  =%
%TCIMACRO{\dsum _{ij\in E\left(  G\right)  }}%
%BeginExpansion
{\displaystyle\sum_{ij\in E\left(  G\right)  }}
%EndExpansion
x_{i}x_{j}%
\]
and more generally,
\[
L_{r}\left(  A,\mathbf{x}\right)  =%
%TCIMACRO{\dsum }%
%BeginExpansion
{\displaystyle\sum}
%EndExpansion
\left\{  x_{i_{1}}\cdots\text{ }x_{i_{r}}:\text{ the set }\left\{
i_{1},\ldots,i_{r}\right\}  \text{ induces an }r\text{-clique in }G\right\}
.
\]

To finish the proof of Theorem \ref{mainTh} we show that $G$ is a complete
graph and $L_{r}\left(  A,\mathbf{x}\right)  =\varphi_{r}\left(  c\right)  .$

\subsection*{Proof that $G$ is a complete graph}

For convenience we first outline this part of the proof. Write $\overline{G}$
for the complement of $G$ and $E\left(  \overline{G}\right)  $ for the edge
set of $\overline{G}.$ We assume that $G$ is not complete and reach a
contradiction by the following major steps:

- if $ij\in E\left(  \overline{G}\right)  ,$ then $C_{i}\neq C_{j}$ - Claim
\ref{cl1};

- if $ij\in E\left(  G\right)  ,$ then $D_{ij}<\lambda$ - Claim \ref{cl2};

- $\overline{G}$ is triangle-free - Claim \ref{cl4};

- $\overline{G}$ is bipartite - Claims \ref{cl5} and \ref{cl6};

- $G$ contains induced $4$-cycles - Claim \ref{cl7};

- $G$ contains no induced $4$-cycles - Claim \ref{cl7.1}.\bigskip

Now the details.

\begin{claim}
\label{cl1}If $ij\in E\left(  \overline{G}\right)  ,$ then $C_{i}\neq C_{j}$.
\end{claim}

\begin{proof}
Assume that $ij\in E\left(  \overline{G}\right)  $ and $C_{i}=C_{j}.$ Let
$\mathbf{y}=\left(  x_{1}+\Delta_{1},\ldots,x_{n}+\Delta_{n}\right)  ,$ where
\[
\Delta_{i}=-x_{i},\text{ \ \ }\Delta_{j}=x_{i},\text{ \ \ and \ \ }\Delta
_{l}=0\text{\ for }l\in\left[  n\right]  \backslash\left\{  i,j\right\}  .
\]
Clearly, $\mathbf{y}\in\mathcal{X}\left(  n\right)  ;$ Taylor's expansion
(\ref{Tay2}) gives
\[
L_{2}\left(  A,\mathbf{y}\right)  -L_{2}\left(  A,\mathbf{x}\right)
=C_{j}x_{i}-C_{i}x_{i}=0;
\]
thus, $\left(  A,\mathbf{y}\right)  \in\mathcal{S}_{n}\left(  c\right)  .$
Taylor's expansion (\ref{Tayr}) and Lagrange's condition (\ref{lag1}) give\
\[
L_{r}\left(  A,\mathbf{y}\right)  -L_{r}\left(  A,\mathbf{x}\right)
=D_{j}x_{i}-D_{i}x_{i}=\mu\left(  x_{i}-x_{i}\right)  +\lambda\left(
C_{j}-C_{i}\right)  x_{i}=0,
\]
contradicting Claim \ref{cl0} as the $i$th entry of $\mathbf{y}$ is zero. The
proof of Claim \ref{cl1} is completed.
\end{proof}

\begin{claim}
\label{cl2}If $ij\in E\left(  G\right)  ,$ then $D_{ij}<\lambda.$
\end{claim}

\begin{proof}
Assume that $ij\in E\left(  G\right)  $ and $D_{ij}\geq\lambda.$ Select $pq\in
E\left(  \overline{G}\right)  ;$ by Claim \ref{cl1} suppose that $C_{p}>C_{q}%
$. For every $\alpha\in\left(  0,x_{q}\right)  ,$ let $\mathbf{y}_{\alpha
}=\left(  y_{1},\ldots,y_{n}\right)  ,$ where
\[
y_{p}=x_{p}+\alpha,\text{ \ \ }y_{q}=x_{q}-\alpha,\text{ \ \ and \ \ }%
y_{l}=x_{l}\text{ for all }l\in\left[  n\right]  \backslash\left\{
p,q\right\}  .
\]
Let
\begin{equation}
f\left(  \alpha\right)  =\frac{\left(  C_{q}-C_{p}\right)  \alpha}{y_{i}y_{j}%
}. \label{eq3}%
\end{equation}
and define the $n\times n$ matrix $B_{\alpha}=\left(  b_{rs}\right)  $ by%
\[
b_{ij}=b_{ji}=1+f\left(  \alpha\right)  ,\text{ \ \ and \ \ }b_{rs}%
=a_{rs}\text{ for }\left\{  r,s\right\}  \neq\left\{  i,j\right\}  .
\]

For $\alpha$ sufficiently small, $-1<f\left(  \alpha\right)  <0,$ and so
$B_{\alpha}\in\mathcal{A}\left(  n\right)  $ and $\mathbf{y}_{\alpha}%
\in\mathcal{X}\left(  n\right)  .$ Taylor's expansion (\ref{Tay2}) and
equation (\ref{eq3}) give
\begin{align*}
L_{2}\left(  B_{\alpha},\mathbf{y}_{\alpha}\right)  -L_{2}\left(
A,\mathbf{x}\right)   &  =L_{2}\left(  B_{\alpha},\mathbf{y}_{\alpha}\right)
-L_{2}\left(  A,\mathbf{y}_{\alpha}\right)  +L_{2}\left(  A,\mathbf{y}%
_{\alpha}\right)  -L_{2}\left(  A,\mathbf{x}\right)  \\
&  =f\left(  \alpha\right)  y_{i}y_{j}+\alpha\left(  C_{p}-C_{q}\right)  =0;
\end{align*}
thus, $\left(  B_{\alpha},\mathbf{y}_{\alpha}\right)  \in\mathcal{S}%
_{n}\left(  c\right)  .$ 

Taylor's expansion (\ref{Tayr}), Lagrange's condition (\ref{lag1}), and
equation (\ref{eq3}) give%
\begin{align*}
L_{r}\left(  B_{\alpha},\mathbf{y}_{\alpha}\right)  -L_{r}\left(
A,\mathbf{x}\right)   &  =L_{r}\left(  B_{\alpha},\mathbf{y}_{\alpha}\right)
-L_{r}\left(  A,\mathbf{y}_{\alpha}\right)  +L_{r}\left(  A,\mathbf{y}%
_{\alpha}\right)  -L_{r}\left(  A,\mathbf{x}\right)  \\
&  =D_{p}\alpha-D_{q}\alpha+D_{ij}f\left(  \alpha\right)  y_{i}y_{j}%
=\lambda\left(  C_{p}-C_{q}\right)  \alpha-D_{ij}\left(  C_{p}-C_{q}\right)
\alpha\\
&  =\alpha\left(  C_{p}-C_{q}\right)  \left(  \lambda-D_{ij}\right)  .
\end{align*}
Since $L_{r}\left(  B_{\alpha},\mathbf{y}_{\alpha}\right)  \geq L_{r}\left(
A,\mathbf{x}\right)  ,$ $\alpha\left(  C_{p}-C_{q}\right)  >0,$ and
$D_{ij}\geq\lambda,$ we see that $L_{r}\left(  B_{\alpha},\mathbf{y}_{\alpha
}\right)  =L_{r}\left(  A,\mathbf{x}\right)  .$

If there exists $\alpha\in\left(  0,x_{q}\right)  $ such that $a_{ij}+f\left(
\alpha\right)  =0,$ then the $\left(  0,1\right)  $-matrix $B_{\alpha}$ has
more zero entries than $A,$ contradicting the main assumption. On the other
hand, if $a_{ij}+f\left(  \alpha\right)  >0$ for all $\alpha\in\left(
0,x_{q}\right)  ,$ then $q\notin\left\{  i,j\right\}  $ and the definitions of
$f\left(  \alpha\right)  ,$ $B_{\alpha},$ and $\mathbf{y}_{\alpha}$ make sense
for $\alpha=x_{q}$ as well. Letting $\alpha=x_{q},$ we obtain $y_{q}=0,$
contradicting Claim \ref{cl0} and completing the proof of Claim \ref{cl2}.
\end{proof}

\begin{claim}
\label{cl4}The graph $\overline{G}$ is triangle-free.
\end{claim}

\begin{proof}
Assume the assertion false and let $i,j,k\in\left[  n\right]  $ be such that
$ij,ik,jk\in E\left(  \overline{G}\right)  .$ Let the line given by
\begin{equation}
\left(  C_{i}-C_{k}\right)  x+\left(  C_{j}-C_{k}\right)  y=0 \label{eq4}%
\end{equation}
intersect the triangle formed by the lines $x=-x_{i}$, $y=-x_{j}$%
,\ $x+y=x_{k}$ at some point $\left(  \alpha,\beta\right)  .$ Let
$\mathbf{y}=\left(  x_{1}+\Delta_{1},\ldots,x_{n}+\Delta_{n}\right)  ,$ where%
\[
\Delta_{i}=\alpha,\text{ \ \ }\Delta_{j}=\beta,\text{ \ \ }\Delta_{k}%
=-\alpha-\beta,\text{ \ \ and \ \ }\Delta_{l}=0\text{ for }l\in\left[
n\right]  \backslash\left\{  i,j,k\right\}  .
\]
Clearly, $\mathbf{y}\in\mathcal{X}\left(  n\right)  ;$ Taylor's expansion
((\ref{Tay2}) and equation (\ref{eq4}) give
\[
L_{2}\left(  A,\mathbf{y}\right)  -L_{2}\left(  A,\mathbf{x}\right)
=C_{i}\alpha+C_{j}\beta-C_{k}\left(  \alpha+\beta\right)  =0;
\]
thus $\left(  A,\mathbf{y}\right)  \in\mathcal{S}_{n}\left(  c\right)  .$
Taylor's expansion (\ref{Tayr}), Lagrange's condition (\ref{lag1}), and
equation (\ref{eq4}) give
\begin{align*}
L_{r}\left(  A,\mathbf{y}\right)  -L_{r}\left(  A,\mathbf{x}\right)   &
=D_{i}\alpha+D_{j}\beta-D_{k}\left(  \alpha+\beta\right) \\
&  =\mu\left(  \alpha+\beta-\alpha-\beta\right)  +\lambda\left(  \left(
C_{i}-C_{k}\right)  \alpha+\left(  C_{j}-C_{k}\right)  \beta\right)  =0,
\end{align*}
contradicting Claim \ref{cl0} as $\mathbf{y}$ has a zero entry. The proof of
Claim \ref{cl4} is completed.
\end{proof}

Using the following claim, we shall prove that $\overline{G}$ is a specific
bipartite graph.

\begin{claim}
\label{cl5}Let the vertices $i,j,k$ satisfy $ij\in E\left(  G\right)  ,$
$ik\in E\left(  \overline{G}\right)  ,$ $jk\in E\left(  \overline{G}\right)
.$ Then
\[
\left(  C_{i}-C_{k}\right)  \left(  C_{j}-C_{k}\right)  >0.
\]

\end{claim}

\begin{proof}
Note first that by Claim \ref{cl1} we have $C_{i}\neq C_{k}$ and $C_{j}\neq
C_{k}.$ Consider the hyperbola defined by
\begin{equation}
\left(  C_{i}-C_{k}\right)  x+\left(  C_{j}-C_{k}\right)  y+xy=0, \label{hyp}%
\end{equation}
and write $H$ for its branch containing the origin. Obviously $\left(
C_{i}-C_{k}\right)  \left(  C_{j}-C_{k}\right)  <0$ implies that $\alpha
\beta>0$ for all $\left(  \alpha,\beta\right)  \in H$.

Suppose $\left(  \alpha,\beta\right)  \in H$ is sufficiently close to the
origin and let $\mathbf{y}=\left(  x_{1}+\Delta_{1},\ldots,x_{n}+\Delta
_{n}\right)  ,$ where%
\[
\Delta_{i}=\alpha,\text{ \ \ }\Delta_{j}=\beta,\text{ \ \ }\Delta_{k}%
=-\alpha-\beta,\text{ \ \ and \ \ }\Delta_{l}=0\text{ for }l\in\left[
n\right]  \backslash\left\{  i,j,k\right\}  .
\]
Clearly, $\mathbf{y}\in\mathcal{X}\left(  n\right)  ;$ Taylor's expansion
(\ref{Tay2}) and equation (\ref{hyp}) give%
\[
L_{2}\left(  A,\mathbf{y}\right)  -L_{2}\left(  A,\mathbf{x}\right)
=C_{i}\alpha+C_{j}\beta-C_{k}\left(  \alpha+\beta\right)  +\alpha\beta=0;
\]
thus $\left(  A,\mathbf{y}\right)  \in\mathcal{S}_{n}\left(  c\right)  .$
Taylor's expansion (\ref{Tayr}), Lagrange's condition (\ref{lag1}), and
equation (\ref{hyp}) give%
\begin{align*}
L_{r}\left(  A,\mathbf{y}\right)  -L_{r}\left(  A,\mathbf{x}\right)   &
=D_{i}\alpha+D_{j}\beta-D_{k}\left(  \alpha+\beta\right)  +D_{ij}\alpha\beta\\
&  =\lambda\left(  C_{i}\alpha+C_{j}\beta-C_{k}\left(  \alpha+\beta\right)
\right)  +D_{ij}\alpha\beta=\left(  D_{ij}-\lambda\right)  \alpha\beta.
\end{align*}
Since $D_{ij}<\lambda$ and $L_{r}\left(  A,\mathbf{y}\right)  \leq
L_{r}\left(  A,\mathbf{x}\right)  ,$ we see that $\alpha\beta<0.$ Thus,
$\left(  C_{i}-C_{k}\right)  \left(  C_{j}-C_{k}\right)  >0,$ completing the
proof of Claim \ref{cl5}.
\end{proof}

\begin{claim}
\label{cl6} $\overline{G}$ is a bipartite graph and its vertex classes $U^{+}$
and $U^{-}$ can be selected so that $C_{u}>C_{w}$ for all $u\in U^{+}$ and
$w\in U^{-}$ such that $uw\in E\left(  \overline{G}\right)  .$
\end{claim}

\begin{proof}
Since $C_{i}\neq C_{j}$ for every $ij\in E\left(  \overline{G}\right)  ,$ if
$\overline{G}$ has an odd cycle, there exist three consecutive vertices
$i,k,j$ along the cycle such that $\left(  C_{i}-C_{k}\right)  \left(
C_{j}-C_{k}\right)  <0.$ Since $\overline{G}$ is triangle-free, $ij\in
E\left(  G\right)  ;$ hence the existence of the vertices $i,j,k$ contradicts
Claim \ref{cl5}. Thus, $\overline{G}$ is bipartite.

Claim \ref{cl5} implies that for every $u\in\left[  n\right]  ,$ the value
$C_{u}-C_{v}$ has the same sign for every $v$ such that $uv\in E\left(
\overline{G}\right)  .$ Let $U^{+}$ be the set of vertices for which this sign
is positive, and let $U^{-}=\left[  n\right]  \backslash U^{+}.$ Clearly, for
every $uv\in E\left(  \overline{G}\right)  ,$ if $u\in U^{+},$ then $v\in
U^{-},$ and if $u\in U^{-},$ then $v\in U^{+}$. Hence, $U^{+}$ and $U^{-}$
partition properly the vertices of $\overline{G},$ completing the proof of
Claim \ref{cl5}.
\end{proof}

Hereafter we suppose that the vertex classes $U^{+}$ and $U^{-}$ of
$\overline{G}$ are selected to satisfy the condition of Claim \ref{cl6}. Note
that $U^{+}$ and $U^{-}$ induce complete graphs in $G.$

\begin{claim}
\label{cl7} $G$ contains an induced $4$-cycle.
\end{claim}

\begin{proof}
Assume the assertion false. For every vertex $u,$ write $N\left(  u\right)  $
for the set of its neighbors in the vertex class opposite to its own class. 

If there exist $u,v\in U^{+}$ such that $\ N\left(  u\right)  \backslash
N\left(  v\right)  \neq\varnothing$ and $N\left(  v\right)  \backslash
N\left(  u\right)  \neq\varnothing,$ taking $x\in N\left(  u\right)
\backslash N\left(  v\right)  $ and $y\in N\left(  v\right)  \backslash
N\left(  u\right)  ,$ we see that $\left\{  x,y,u,v\right\}  $ induces a
$4$-cycle in $G;$ thus we will assume that $N\left(  u\right)  \subset
N\left(  v\right)  $ or $N\left(  v\right)  \subset N\left(  u\right)  $ for
every $u,v\in U^{+}.$ This condition implies that there is a vertex $u_{1}\in
U^{+}$ such that $N\left(  v\right)  \subset N\left(  u_{1}\right)  $ for
every $v\in U^{+}.$ By symmetry, there is a vertex $u_{2}\in U^{-}$ such that
$N\left(  v\right)  \subset N\left(  u_{2}\right)  $ for every $v\in U^{-}.$

If $N\left(  u_{1}\right)  \neq U^{-}$ and $N\left(  u_{2}\right)  \neq
U^{+},$ take $x\in U^{-}\backslash N\left(  u_{1}\right)  $ and $y\in
U^{+}\backslash N\left(  u_{2}\right)  ,$ and note that $N\left(  x\right)
=\varnothing$ and $N\left(  y\right)  =\varnothing.$ Hence, adding the edge
$xy$ to $E\left(  G\right)  ,$ we see that $L_{r}\left(  A,\mathbf{x}\right)
$ remains the same, while $L_{2}\left(  A,\mathbf{x}\right)  $ increases,
contradicting that $\varphi_{r}\left(  n,c\right)  $ is increasing in $c$
(Proposition \ref{pro1}). Thus, either $N\left(  u_{1}\right)  =U^{-}$ or
$N\left(  u_{2}\right)  =U^{+},$ so one of the vertices $u_{1}$ or $u_{2}$ is
connected to every vertex other than itself.

By symmetry, suppose that the vertex $n$ is connected to every vertex of $G$
other than itself. Set $\mathbf{y}=\left(  x_{1},\ldots,x_{n-1}\right)  $ and
let $B$ be the principal submatrix of $A$ in the first $\left(  n-1\right)  $
columns. Since%
\begin{align}
x_{1}+\cdots+x_{n-1}  &  =1-x_{n},\label{con1}\\
L_{2}\left(  B,\mathbf{y}\right)   &  =c-x_{n}\left(  1-x_{n}\right)  ,
\label{con2}%
\end{align}
and
\[
L_{r}\left(  A,\mathbf{x}\right)  =x_{n}L_{r-1}\left(  B,\mathbf{y}\right)
+L_{r}\left(  B,\mathbf{y}\right)  ,
\]
we see that $x_{n}L_{r-1}\left(  B,\mathbf{y}\right)  +L_{r}\left(
B,\mathbf{y}\right)  $ is minimum subject to (\ref{con1}) and (\ref{con2}).
Since $B\in\mathcal{A}\left(  n-1\right)  $, by the main assumption, both
$L_{r-1}\left(  B,\mathbf{z}\right)  $ and $L_{r}\left(  B,\mathbf{z}\right)
$ attain a minimum on a complete graph $H$ and for the same vector
$\mathbf{z}$. Since $n$ is joined to every vertex of $H,$ the minimum
$\varphi_{r}\left(  n,c\right)  $ is attained on a complete graph too, a
contradiction completing the proof of Claim \ref{cl7}.
\end{proof}

For convenience, an induced $4$-cycle in $G$ will be denoted by a quadruple
$\left(  i,j,k,l\right)  ,$ where $i,j,k,l$ are the vertices of the cycle,
arranged so that $i,j\in U^{+},$ $k,l\in U^{-},$ $ik\notin E\left(  G\right)
,$ and $jl\notin E\left(  G\right)  .$

\begin{claim}
\label{cl7.1} If $\left(  i,j,k,l\right)  $ is an induced $4$-cycle in $G,$
then $D_{ij}+D_{kl}<D_{jk}+D_{li}.$
\end{claim}

\begin{proof}
Indeed, let $L$ be the line defined by
\begin{equation}
\left(  C_{i}-C_{k}\right)  x+\left(  C_{j}-C_{l}\right)  y=0.\label{lin}%
\end{equation}
Since $i,j\in U^{+}$ and $k,l\in U^{-},$ we have $C_{i}>C_{k}$ and
$C_{j}>C_{l};$ thus $xy<0$ for all $\left(  x,y\right)  \in L$. Suppose that
$\alpha\in\left(  0,x_{k}\right)  ,$ $\beta\in\left(  -x_{j},0\right)  ,$ and
$\left(  \alpha,\beta\right)  \in L.$ Let $\mathbf{y}_{\alpha}=\left(
x_{1}+\Delta_{1},\ldots,x_{n}+\Delta_{n}\right)  ,$ where%
\[
\Delta_{i}=\alpha,\text{ \ \ }\Delta_{j}=\beta,\text{ \ \ }\Delta_{k}%
=-\alpha,\text{ \ \ }\Delta_{l}=-\beta,\text{ \ \ and \ \ }\Delta_{h}=0\text{
for }h\in\left[  n\right]  \backslash\left\{  i,j,k,l\right\}  .
\]
Clearly, $\mathbf{y}_{\alpha}\in\mathcal{X}\left(  n\right)  ;$ Taylor's
expansion (\ref{Tay2}) and equation (\ref{lin}) give%
\[
L_{2}\left(  A,\mathbf{y}_{\alpha}\right)  -L_{2}\left(  A,\mathbf{x}\right)
=\left(  C_{i}-C_{k}\right)  \alpha+\left(  C_{j}-C_{l}\right)  \beta
+\alpha\beta-\alpha\beta+\alpha\beta-\alpha\beta=0;
\]
thus $\left(  A,\mathbf{y}_{\alpha}\right)  \in\mathcal{S}_{n}\left(
c\right)  .$ Taylor's expansion (\ref{Tayr}), Lagrange's condition
(\ref{lag1}), and equation (\ref{lin}) give%
\begin{align*}
L_{r}\left(  A,\mathbf{y}_{\alpha}\right)  -L_{r}\left(  A,\mathbf{x}\right)
&  =D_{i}\alpha+D_{j}\beta-D_{k}\alpha-D_{l}\beta+\left(  D_{ij}-D_{jk}%
+D_{kl}-D_{li}\right)  \alpha\beta\\
&  =\lambda\left(  C_{i}\alpha+C_{j}\beta-C_{k}\alpha-C_{l}\beta\right)
+\left(  D_{ij}-D_{jk}+D_{kl}-D_{li}\right)  \alpha\beta\\
&  =\left(  D_{ij}-D_{jk}+D_{kl}-D_{li}\right)  \alpha\beta.
\end{align*}
Since $L_{r}\left(  A,\mathbf{y}_{\alpha}\right)  \geq L_{r}\left(
A,\mathbf{x}\right)  $ and $\alpha\beta<0,$ we find that $D_{ij}+D_{kl}\leq
D_{jk}+D_{li}.$ If $D_{ij}+D_{kl}=D_{jk}+D_{li}$, setting%
\[
\alpha=\min\left\{  x_{k},\frac{C_{j}-C_{l}}{C_{i}-C_{k}}x_{j}\right\}  ,
\]
we see that $L_{r}\left(  A,\mathbf{y}_{\alpha}\right)  =L_{r}\left(
A,\mathbf{x}\right)  $ and either the $k$th or the $j$th entry of
$\mathbf{y}_{\alpha}$ is zero, contradicting Claim \ref{cl0}. Hence,
$D_{ij}+D_{kl}<D_{jk}+D_{li},$ completing the proof of Claim \ref{cl7.1}.
\end{proof}

Select an induced $4$-cycle $\left(  i,j,k,l\right)  $ and let us investigate
$D_{ij},D_{kl},D_{jk},$ and $D_{li}$ in the light of Claim \ref{cl7.1}. We
have
\begin{align*}
D_{ij}  &  =\sum\left\{  x_{i_{1}}\cdots x_{i_{r-2}}:\left\{  i,j,i_{1}%
,\ldots,i_{r-2}\right\}  \text{ induces an }r\text{-clique}\right\}  ,\\
D_{kl}  &  =\sum\left\{  x_{i_{1}}\cdots x_{i_{r-2}}:\left\{  k,l,i_{1}%
,\ldots,i_{r-2}\right\}  \text{ induces an }r\text{-clique}\right\}  ,\\
D_{jk}  &  =\sum\left\{  x_{i_{1}}\cdots x_{i_{r-2}}:\left\{  j,k,i_{1}%
,\ldots,i_{r-2}\right\}  \text{ induces an }r\text{-clique}\right\}  ,\\
D_{li}  &  =\sum\left\{  x_{i_{1}}\cdots x_{i_{r-2}}:\left\{  j,k,i_{1}%
,\ldots,i_{r-2}\right\}  \text{ induces an }r\text{-clique}\right\}  .
\end{align*}
First note that if a product $x_{i_{1}}\cdots x_{i_{r-2}}$ is present in any
of the above sums, then $\left\{  i_{1},\ldots,i_{r-2}\right\}  \cap\left\{
i,j,k,l\right\}  \neq\varnothing.$ Also, a product $x_{i_{1}}\cdots
x_{i_{r-2}}$ is present in both $D_{ij}$ and $D_{kl}$ exactly when it is
present in both $D_{jk}$ and $D_{li}.$ Hence, Claim \ref{cl7.1} implies that
there exists a set $\left\{  i_{1},\ldots,i_{r-2}\right\}  $ such that either
$\left\{  j,k,i_{1},\ldots,i_{r-2}\right\}  $ or $\left\{  i,l,i_{1}%
,\ldots,i_{r-2}\right\}  $ induces an $r$-clique, but neither $\left\{
i,j,i_{1},\ldots,i_{r-2}\right\}  $ nor $\left\{  k,l,i_{1},\ldots
,i_{r-2}\right\}  $ induces an $r$-clique. This is a contradiction for $r=3,$
as either $\left\{  p,i,j\right\}  $ or $\left\{  p,k,l\right\}  $ induces a
triangle for every vertex $p\notin\left\{  i,j,k,l\right\}  .$

Let now $r=4$. We shall reach a contradiction by proving that $D_{ij}%
+D_{kl}\geq D_{jk}+D_{li}.$ Let $D_{ij}^{\ast}$ be the sum of all products
$x_{p}x_{q}$ present in $D_{ij}$ but not present in any of $D_{jk}%
,D_{kl},D_{il}.$ Defining the sums $D_{jk}^{\ast},$ $D_{kl}^{\ast},$ and
$D_{il}^{\ast}$ likewise, we see that
\[
D_{ij}+D_{kl}-D_{jk}-D_{li}=D_{ij}^{\ast}+D_{kl}^{\ast}-D_{jk}^{\ast}%
-D_{li}^{\ast},
\]
so it suffices to prove $D_{ij}^{\ast}+D_{kl}^{\ast}-D_{jk}^{\ast}%
-D_{li}^{\ast}\geq0.$ To this end, write $\Gamma\left(  u\right)  $ for the
set of neighbors of a vertex $u$ and set
\begin{align*}
A &  =\Gamma\left(  i\right)  \backslash\Gamma\left(  k\right)  ,\text{
\ \ }B=\Gamma\left(  j\right)  \backslash\Gamma\left(  l\right)  ,\text{
\ \ }X=A\cap B,\\
C &  =\Gamma\left(  k\right)  \backslash\Gamma\left(  i\right)  ,\text{
\ \ }D=\Gamma\left(  l\right)  \backslash\Gamma\left(  j\right)  ,\text{
\ \ }Y=C\cap D,\\
a &  =\sum_{p\in A}x_{p},\text{ \ \ }b=\sum_{p\in B}x_{p},\text{ \ \ }%
c=\sum_{p\in C}x_{p},\text{ \ \ }d=\sum_{p\in D}x_{p},\text{ \ \ }x=\sum_{p\in
X}x_{p},\text{ \ \ }y=\sum_{p\in Y}x_{p}.
\end{align*}

Observe that $A,B$ and $X$ are subsets of $U^{+}\backslash\left\{
i,j\right\}  ,$ while $C,D$ and $Y$ are subsets of $U^{-}\backslash\left\{
k,l\right\}  .$ For reader's sake, here is an alternative view on $A,B,C,D,X,$
and $Y$:
\begin{align*}
A\backslash X  &  =\Gamma\left(  i\right)  \cap\Gamma\left(  j\right)
\cap\Gamma\left(  l\right)  \backslash\Gamma\left(  k\right)  ,\text{
\ \ }B\backslash X=\Gamma\left(  i\right)  \cap\Gamma\left(  j\right)
\cap\Gamma\left(  k\right)  \backslash\Gamma\left(  l\right)  ,\\
C\backslash Y  &  =\Gamma\left(  k\right)  \cap\Gamma\left(  l\right)
\cap\Gamma\left(  j\right)  \backslash\Gamma\left(  i\right)  ,\text{
\ \ }D\backslash Y=\Gamma\left(  k\right)  \cap\Gamma\left(  l\right)
\cap\Gamma\left(  i\right)  \backslash\Gamma\left(  j\right)  ,\\
X  &  =\Gamma\left(  i\right)  \cap\Gamma\left(  j\right)  \backslash\left(
\Gamma\left(  k\right)  \cup\Gamma\left(  l\right)  \right)  ,\text{
\ \ \ }Y=\Gamma\left(  k\right)  \cap\Gamma\left(  l\right)  \backslash\left(
\Gamma\left(  i\right)  \cup\Gamma\left(  j\right)  \right)  ,
\end{align*}

Let the product $x_{p}x_{q}$ be present in $D_{jk}^{\ast.};$ thus $\left\{
j,k,p,q\right\}  $ induces an $4$-clique, but neither $\left\{
i,j,p,q\right\}  $ nor $\left\{  k,l,p,q\right\}  $ induces an $4$-clique.
Clearly, $p$ and $q$ belong to different vertex classes of $\overline{G},$ say
$p\in U^{+}$ and $q\in U^{-}.$ Since $i,j,$ and $k$ are joined to $p,$ we must
have $pl\notin E\left(  G\right)  ,$ and so $p\in B\backslash X;$ likewise we
find that $q\in C\backslash Y$. Thus
\begin{equation}
D_{jk}^{\ast}\leq\sum_{u\in B\backslash X}x_{u}\sum_{u\in C\backslash Y}%
x_{u}=\left(  b-x\right)  \left(  c-y\right)  ,\label{mi1}%
\end{equation}
and by symmetry,%
\begin{equation}
D_{il}^{\ast}\leq\sum_{u\in A\backslash X}x_{u}\sum_{u\in D\backslash Y}%
x_{u}=\left(  a-x\right)  \left(  d-y\right)  .\label{mi2}%
\end{equation}

For every pair $\left(  p,q\right)  $ satisfying
\[
p\in X,\text{ }q\in B\backslash X,\text{ \ \ or \ \ }p\in A\backslash X,\text{
}q\in X,\text{ \ \ or \ \ }p\in A\backslash X,\text{ }q\in B\backslash X,
\]
we see that $\left\{  i,j,p,q\right\}  $ induces an $4$-clique, but $p$ is not
joined to $k$ and $q$ is not joined to $l;$ thus $x_{p}x_{q}$ is present in
$D_{ij}^{\ast}.$ Therefore,
\begin{equation}
D_{ij}^{\ast}\geq\sum_{u\in X}x_{u}\sum_{u\in B\backslash X}x_{u}+\sum_{u\in
A\backslash X}x_{u}\sum_{u\in X}x_{u}+\sum_{u\in A\backslash X}x_{u}\sum_{u\in
B\backslash X}x_{u}=ab-x^{2},\label{mi3}%
\end{equation}
and by symmetry,%
\begin{equation}
D_{kl}^{\ast}\geq\sum_{u\in Y}x_{u}\sum_{u\in D\backslash Y}x_{u}+\sum_{u\in
C\backslash Y}x_{u}\sum_{u\in Y}x_{u}+\sum_{u\in C\backslash Y}x_{u}\sum_{u\in
D\backslash Y}x_{u}=cd-y^{2}.\label{mi4}%
\end{equation}
Now adding (\ref{mi3}) and (\ref{mi4}), and subtracting (\ref{mi1}) and
(\ref{mi2}), we obtain%
\begin{align*}
D_{ij}^{\ast}+D_{kl}^{\ast}-D_{jk}^{\ast}-D_{li}^{\ast} &  \geq ab-x^{2}%
+cd-y^{2}-\left(  b-x\right)  \left(  c-y\right)  -\left(  a-x\right)  \left(
d-y\right)  \\
&  =\left(  a-c\right)  \left(  b-d\right)  +x\left(  c+d\right)  +y\left(
a+b\right)  -\left(  x+y\right)  ^{2}.
\end{align*}
Hence, using $x\leq\min\left(  a,b\right)  ,$ $y\leq\min\left(  c,d\right)  ,$
and the inequalities
\begin{align*}
a-c  & =\sum_{u\in\Gamma\left(  i\right)  \backslash\Gamma\left(  k\right)
}x_{u}+\sum_{u\in\Gamma\left(  i\right)  \cap\Gamma\left(  k\right)  }%
x_{u}-\sum_{u\in\Gamma\left(  k\right)  \backslash\Gamma\left(  i\right)
}x_{u}-\sum_{u\in\Gamma\left(  i\right)  \cap\Gamma\left(  k\right)  }%
x_{u}=C_{i}-C_{k}>0,\\
b-d  & =\sum_{u\in\Gamma\left(  j\right)  \backslash\Gamma\left(  l\right)
}x_{u}+\sum_{u\in\Gamma\left(  j\right)  \cap\Gamma\left(  l\right)  }%
x_{u}-\sum_{u\in\Gamma\left(  l\right)  \backslash\Gamma\left(  j\right)
}x_{u}-\sum_{u\in\Gamma\left(  j\right)  \cap\Gamma\left(  l\right)  }%
x_{u}=C_{j}-C_{l}>0,
\end{align*}
Lemma \ref{le1} implies that $D_{ij}^{\ast}+D_{kl}^{\ast}-D_{jk}^{\ast}%
-D_{li}^{\ast}\geq0,$ as required.

This finishes the proof that $G$ is a complete graph for $r=3,4$.

\subsection*{Proof of $L_{r}\left(  A,\mathbf{x}\right)  =\varphi_{r}\left(
c\right)  $}

We know now that $G$ is a complete graph. We have to show that $n=\xi\left(
c\right)  $ and $\left(  x_{1},\ldots,x_{n}\right)  =\left(  x,\ldots
,x,y\right)  ,$ where $x$ and $y$ are given by (\ref{sol}). Our proof is based
on the following assertion.$^{{}}$

\begin{claim}
\label{cl8} Let $x_{3}\geq x_{2}\geq x_{1}>0$ be real numbers satisfying%
\begin{align}
x_{1}+x_{2}+x_{3}  &  =a,\label{cons1}\\
x_{1}x_{2}+x_{2}x_{3}+x_{3}x_{1}  &  =b, \label{cons1.1}%
\end{align}
and let $x_{1}x_{2}x_{3}$ be minimum subject to (\ref{cons1}) and
(\ref{cons1.1}). Then $x_{2}=x_{3}$.
\end{claim}

\begin{proof}
First note that the hypothesis implies that
\begin{equation}
a^{2}/4<b\leq a^{2}/3. \label{condb}%
\end{equation}
Indeed, the second of these inequalities follows from Maclaurin's inequality;
assume for a contradiction that the first one fails. Then, selecting a
sufficiently small $\varepsilon>0$ and setting
\[
y_{1}=\varepsilon,\text{ }y_{2}=\frac{a-\varepsilon-\sqrt{\left(
a+\varepsilon\right)  ^{2}-4\left(  b+\varepsilon^{2}\right)  }}{2},\text{
}y_{3}=\frac{a-\varepsilon+\sqrt{\left(  a+\varepsilon\right)  ^{2}-4\left(
b+\varepsilon^{2}\right)  }}{2},
\]
we see that $y_{1}$, $y_{2}$, $y_{3}$ satisfy (\ref{cons1}), (\ref{cons1.1}),
and
\[
y_{1}y_{2}y_{3}=\varepsilon\left(  b-a\varepsilon+\varepsilon^{2}\right)
<\varepsilon b.
\]
Thus, $\min x_{1}x_{2}x_{3}$, subject to (\ref{cons1}) and (\ref{cons1.1}),
cannot be attained for positive $x_{1}$, $x_{2}$, $x_{3},$ a contradiction,
completing the proof of (\ref{condb}).

By Lagrange's method there exist $\eta$ and $\theta$ such that%
\begin{align*}
x_{1}x_{2}  &  =\eta+\theta\left(  x_{1}+x_{2}\right)  =\eta+\theta\left(
a-x_{3}\right) \\
x_{1}x_{3}  &  =\eta+\theta\left(  x_{1}+x_{3}\right)  =\eta+\theta\left(
a-x_{2}\right) \\
x_{2}x_{3}  &  =\eta+\theta\left(  x_{2}+x_{3}\right)  =\eta+\theta\left(
a-x_{1}\right)  .
\end{align*}

If $\theta=0$ we see that $x_{1}=x_{2}=x_{3},$ completing the proof. Suppose
$\theta\neq0$ and assume for a contradiction that $x_{2}<x_{3}.$ We find that
\begin{align*}
x_{1}\left(  x_{3}-x_{2}\right)   &  =\theta\left(  x_{3}-x_{2}\right)  ,\\
x_{2}\left(  x_{3}-x_{1}\right)   &  =\theta\left(  x_{3}-x_{1}\right)  ,
\end{align*}
and so, $x_{1}=x_{2}.$ Solving the system (\ref{cons1},\ref{cons1.1}) with
$x_{1}=x_{2},$ we obtain
\[
x_{3}=\frac{a}{3}+\frac{2}{3}\sqrt{a^{2}-3b}\text{, \ }x_{1}=x_{2}=\frac{a}%
{3}-\frac{1}{3}\sqrt{a^{2}-3b},
\]
implying that
\begin{equation}
x_{1}x_{2}x_{3}=\left(  \frac{a}{3}+\frac{2}{3}\sqrt{a^{2}-3b}\right)  \left(
\frac{a}{3}-\frac{1}{3}\sqrt{a^{2}-3b}\right)  ^{2}. \label{eq1}%
\end{equation}
If $b=a^{2}/3,$ we see that $x_{1}=x_{2}=x_{3},$ completing the proof, so
suppose that $b<a^{2}/3.$ We shall show that $\min x_{1}x_{2}x_{3},$ subject
to (\ref{cons1}) and (\ref{cons1.1}), is smaller than the right-hand side of
(\ref{eq1}). Indeed, setting
\[
y_{1}=\frac{a}{3}-\frac{2}{3}\sqrt{a^{2}-3b},\text{ \ }y_{2}=y_{3}=\frac{a}%
{3}+\frac{1}{3}\sqrt{a^{2}-3b},
\]
in view of (\ref{condb}), we see that $y_{1},y_{2},y_{3}$ satisfy
(\ref{cons1}) and (\ref{cons1.1}). After some algebra we obtain
\[
y_{1}y_{2}y_{3}-x_{1}x_{2}x_{3}=-\frac{4}{27}\left(  a^{2}-3b\right)
^{3/2}<0.
\]
This contradiction completes the proof of Claim \ref{cl8}.
\end{proof}

Claim \ref{cl8} implies that, out of every three entries of $\mathbf{x},$ the
two largest ones are equal; hence all but the smallest entry of $\mathbf{x}$
are equal. Writing $y$ and $x$ for the smallest and largest entries of
$\mathbf{x},$ we see that $x$ and $y$ satisfy
\begin{align*}
\binom{n-1}{2}x^{2}+nxy  &  =c,\\
\left(  n-1\right)  x+y  &  =1,\\
y  &  \leq x,
\end{align*}
and so,%
\[
y=\frac{1}{n}-\sqrt{1-2\frac{n}{n-1}c},\ \ \ x=\frac{1}{n}+\frac{1}{n}%
\sqrt{1-2\frac{n}{n-1}c}.
\]
Since the condition $1-2nc/\left(  n-1\right)  \geq0$ gives
\[
n\geq\frac{1}{1-2c},
\]
and $y>0$ gives
\[
1-2c<\frac{1}{n}+\frac{1}{n^{2}}<\frac{1}{n-1},
\]
we find that $n=\xi\left(  c\right)  ,$ completing the proof of Theorem
\ref{mainTh}.$\hfill\square$

\subsection{\label{pp0}Proof of Proposition \ref{pro0}}

Suppose that $\mathcal{S}_{n}\left(  c\right)  $ is nonempty and that%

\[
A\in\mathcal{A}\left(  n\right)  ,\ \ \ \mathbf{x}\geq0,\text{ \ \ }%
L_{1}\left(  A,\mathbf{x}\right)  =1,\text{ \ \ and \ \ \ }L_{2}\left(
A,\mathbf{x}\right)  =c.
\]
Then
\[
c=\sum_{1\leq i<j\leq n}a_{ij}x_{i}x_{j}\leq\sum_{1\leq i<j\leq n}x_{i}%
x_{j}=\frac{1}{2}\left(  \sum_{i}x_{i}\right)  ^{2}-\frac{1}{2}\sum_{i}%
x_{i}^{2}\leq\frac{n-1}{2n}<\frac{1}{2},
\]
and so, $c<1/2$ and $n\geq1/\left(  1-2c\right)  ;$ thus $n\geq\left\lceil
1/\left(  1-2c\right)  \right\rceil .$

On the other hand, if $c<1/2$ and $n\geq\left\lceil 1/\left(  1-2c\right)
\right\rceil ,$ let $A\in\mathcal{A}\left(  n\right)  $ be the matrix with all
off-diagonal entries equal to $1,$ and let $x,y$ satisfy
\begin{align*}
\binom{n-1}{2}x^{2}+\left(  n-1\right)  xy  &  =c,\\
\left(  n-1\right)  x+y  &  =1.
\end{align*}
Writing $\mathbf{x}$ for the $n$-vector $\left(  x,\ldots,x,y\right)  ,$ we
see that $L_{1}\left(  A,\mathbf{x}\right)  =1$ and $L_{2}\left(
A,\mathbf{x}\right)  =c;$ thus $\mathcal{S}_{n}\left(  c\right)  $ is
nonempty, completing the proof.$\hfill\square$

\section{\label{apx}Upper bounds on $k_{r}\left(  n,m\right)  $}

In this section we prove Theorem \ref{thmeq}. We start with some facts about
Tur\'{a}n graphs.

The $s$-partite Tur\'{a}n graph $T_{s}\left(  n\right)  $ is a complete
$s$-partite graph on $n$ vertices with each vertex class of size $\left\lfloor
n/s\right\rfloor $ or $\left\lceil n/s\right\rceil .$ Setting $t_{s}\left(
n\right)  =e\left(  T_{s}\left(  n\right)  \right)  ,$ after some algebra we
obtain
\[
t_{s}\left(  n\right)  =\frac{s-1}{2s}n^{2}-\frac{t\left(  s-t\right)  }{2s},
\]
where $t$ is the remainder of $n$ $\operatorname{mod}$ $s;$ hence,$_{{}}$%
\begin{equation}
\frac{s-1}{2s}n^{2}-\frac{s}{8}\leq t_{s}\left(  n\right)  \leq\frac{s-1}%
{2s}n^{2}. \label{turest}%
\end{equation}
It is known that the second one of these inequalities can be extended for all
$2\leq r\leq s:_{{}}$%

\begin{equation}
k_{r}\left(  T_{s}\left(  n\right)  \right)  \leq\binom{s}{r}\left(  \frac
{n}{s}\right)  ^{r}. \label{turestr}%
\end{equation}

The Tur\'{a}n graphs play an exceptional role for the function $k_{r}\left(
n,m\right)  :$ indeed, a result of Bollob\'{a}s \cite{Bol76} implies that if
$G$ is a graph with $n$ vertices and $t_{s}\left(  n\right)  $ edges, then
$k_{r}\left(  G\right)  \geq k_{r}\left(  T_{s}\left(  n\right)  \right)  ;$ hence,

\begin{fact}
\label{f1}$k_{r}\left(  n,t_{s}\left(  n\right)  \right)  =k_{r}\left(
T_{s}\left(  n\right)  \right)  .\hfill\square$
\end{fact}

Thus to simplify our presentation, we assume that $n\geq s\geq r\geq3$ are
fixed integers and $m$ is an integer satisfying $t_{s-1}\left(  n\right)
<m\leq t_{s}\left(  n\right)  $.

First we define a class of graphs giving upper bounds on $k_{r}\left(
n,m\right)  .$

\subsection*{The graphs $H\left(  n,m\right)  $}

We shall construct a graph $H\left(  n,m\right)  $ with $n$ vertices and $m$
edges, where $n,s,$ and $m$ satisfy $n\geq s\geq3$ and $t_{s-1}\left(
n\right)  <m\leq t_{s}\left(  n\right)  .$ Note that the construction of
$H\left(  n,m\right)  $ is independent of $r.$

First we define a sequence of graphs $H_{0},\ldots,H_{\left\lfloor
n/s\right\rfloor }$ satisfying
\begin{equation}
t_{s-1}\left(  n\right)  =e\left(  H_{0}\right)  <e\left(  H_{1}\right)
<\cdots<e\left(  H_{\left\lfloor n/s\right\rfloor }\right)  =t_{s}\left(
n\right)  , \label{ehi}%
\end{equation}
and then we construct $H\left(  n,m\right)  $ using $H_{0},\ldots
,H_{\left\lfloor n/s\right\rfloor }.$

\subsection*{The graphs $H_{0},\ldots,H_{\left\lfloor n/s\right\rfloor }$}

For every $0\leq i\leq\left\lfloor n/s\right\rfloor ,$ let $H_{i}$ be the
complete $s$-partite graph with vertex classes $I,V_{1},\ldots,V_{s-1}$ such
that $\left\vert I\right\vert =i$ and
\[
\left\lfloor \left(  n-i\right)  /\left(  s-1\right)  \right\rfloor
=\left\vert V_{1}\right\vert \leq\cdots\leq\left\vert V_{s-1}\right\vert
=\left\lceil \left(  n-i\right)  /\left(  s-1\right)  \right\rceil .
\]

Note that $H_{0}$ is the $\left(  s-1\right)  $-partite Tur\'{a}n graph
$T_{s-1}\left(  n\right)  ,$ but it is convenient to consider it $s$-partite
with an empty vertex class $I$. Note also that $H_{\left\lfloor
n/s\right\rfloor }=T_{s}\left(  n\right)  .$

The transition from $H_{i}$ to $H_{i+1}$ can be briefly summarized as follows:
select $V_{j}$ with $\left\vert V_{j}\right\vert =\left\lceil \left(
n-i\right)  /\left(  s-1\right)  \right\rceil $ and move a vertex $u$ from
$V_{j}$ to $I$.

In particular, we see that
\[
e\left(  H_{i+1}\right)  -e\left(  H_{i}\right)  =\left\lceil \left(
n-i\right)  /\left(  s-1\right)  \right\rceil -i>0,
\]
implying in turn (\ref{ehi}).

\subsection{Constructing $H\left(  n,m\right)  $}

Let $I,V_{1},\ldots,V_{s-1}$ be the vertex classes of $H_{i}.$ Select $V_{j}$
with $\left\vert V_{j}\right\vert =\left\lceil \left(  n-i\right)  /\left(
s-1\right)  \right\rceil $, select a vertex $u\in\left\vert V_{j}\right\vert
,$ let $l=\left\lceil \left(  n-i\right)  /\left(  s-1\right)  \right\rceil
-1,$ and suppose that $V_{j}\backslash\left\{  u\right\}  =\left\{
v_{1},\ldots,v_{l}\right\}  .$ Do the following steps:

\qquad(a) remove all edges joining $u$ to vertices in $I;$

\qquad(b) move $u$ from $V_{j}$ to $I,$ keeping all edges incident to $u;$

\qquad(c) for $m=e\left(  H_{i}\right)  +1,\ldots,e\left(  H_{i+1}\right)  $
join $u$ to $v_{m-e\left(  H_{i}\right)  }$ and write $H\left(  n,m\right)  $
for the resulting graph.

Two observations are in place: first, $e\left(  H\left(  n,m\right)  \right)
=m,$ and second, $H\left(  n,e\left(  H_{i}\right)  \right)  =H_{i}$ for every
$i=1,\ldots,\left\lfloor n/s\right\rfloor .$

Note also that every additional edge in step (c) increases the number of
$r$-cliques by $k_{r-2}\left(  H^{\prime}\right)  ,$ where $H^{^{\prime}}$ is
the fixed graph induced by the set $\left[  n\right]  \backslash\left(  I\cup
V_{j}\right)  .$ We thus make the following

\begin{claim}
\label{cl9}The function $k_{r}\left(  H\left(  n,m\right)  \right)  $
increases linearly in $m$ for $e\left(  H_{i-1}\right)  \leq m\leq e\left(
H_{i}\right)  .$
\end{claim}

We need also the following upper bound on $k_{r}\left(  H_{i}\right)  .$

\begin{claim}
\label{cl12}%
\[
k_{r}\left(  H_{i}\right)  \leq\binom{s-1}{r-1}\left(  \frac{n-i}{s-1}\right)
^{r-1}i+\binom{s-1}{r}\left(  \frac{n-i}{s-1}\right)  ^{r}%
\]

\end{claim}

\begin{proof}
Let $I,V_{1},\ldots,V_{s-1}$ be the vertex classes of $H_{i}.$ Since the sizes
of the sets $V_{1},\ldots,V_{s-1}$ differ by at most $1,$ we see that the set
$V_{1}\cup\ldots\cup V_{s-1}$ induces the Tur\'{a}n graph $T_{s-1}\left(
n-i\right)  .$ Hence a straightforward counting gives
\[
k_{r}\left(  H_{i}\right)  \leq k_{r-1}\left(  T_{s-1}\left(  n-i\right)
\right)  i+k_{r}\left(  T_{s-1}\left(  n-i\right)  \right)  ,
\]
and the claim follows from inequality (\ref{turestr}).
\end{proof}

\subsection{Proof of Theorem \ref{thmeq}}

Assume that $x$ is a real number satisfying
\[
\frac{s-2}{2\left(  s-1\right)  }n^{2}<x\leq\frac{s-1}{2s}n^{2}.
\]
and define the functions $p=p\left(  x\right)  $ and $q=q\left(  x\right)  $
by
\begin{align}
p  &  \geq q,\label{c1}\\
\left(  s-1\right)  p+q  &  =n,\label{c2}\\
\binom{s-1}{2}p^{2}+\left(  s-1\right)  pq  &  =x. \label{c3}%
\end{align}

We note that%
\[
p\left(  x\right)  =\frac{1}{s}\left(  n+\sqrt{n^{2}-\frac{2s}{s-1}x}\right)
,\text{ \ \ }q\left(  x\right)  =\frac{1}{s}\left(  n-\left(  s-1\right)
\sqrt{n^{2}-\frac{2s}{s-1}x}\right)  .
\]
Set
\begin{equation}
f\left(  x\right)  =\binom{s-1}{r}p^{r}+\binom{s-1}{r-1}p^{r-1}q, \label{feq}%
\end{equation}
and note that $f\left(  x\right)  =\varphi_{r}\left(  x/n^{2}\right)  n^{r};$
hence, to prove Theorem \ref{thmeq}, it is enough to show that if
\[
\frac{s-2}{2\left(  s-1\right)  }n^{2}<m\leq\frac{s-1}{2s}n^{2},
\]
then
\begin{equation}
k_{r}\left(  n,m\right)  \leq f\left(  m\right)  +\frac{n^{r}}{n^{2}-2m}.
\label{meq}%
\end{equation}

We first introduce the auxiliary function $\widehat{f}\left(  x\right)  ,$
defined for $x\in\left[  t_{s-1}\left(  n\right)  ,t_{s}\left(  n\right)
\right]  $ by
\[
\widehat{f}\left(  x\right)  =\left\{
\begin{tabular}
[c]{ll}%
$f\left(  x+\frac{s-1}{8}\right)  ,$\bigskip & $\text{if }t_{s-1}\left(
n\right)  <x\leq\frac{s-1}{2s}n^{2}-\frac{s-1}{8};$\\
$f\left(  \frac{s-1}{2s}n^{2}\right)  ,$ & $\text{if }\frac{s-1}{2s}%
n^{2}-\frac{s-1}{8}<x\leq t_{s}\left(  n\right)  .$%
\end{tabular}
\ \right.
\]

To finish the proof of Theorem \ref{thmeq} we first show that
\begin{equation}
k_{r}\left(  H\left(  n,m\right)  \right)  \leq\widehat{f}\left(  m\right)  ,
\label{meq1}%
\end{equation}
and then derive (\ref{meq}) using Taylor's expansion and the fact that
$k_{r}\left(  n,m\right)  \leq k_{r}\left(  H\left(  n,m\right)  \right)  .$

\begin{claim}
\label{cl13}If $m=e\left(  H_{i}\right)  ,$ then%
\[
k_{r}\left(  H_{i}\right)  \leq f\left(  m-t_{s-1}\left(  n-i\right)
+\frac{s-2}{2\left(  s-1\right)  }\left(  n-i\right)  ^{2}\right)  .
\]

\end{claim}

\begin{proof}
Indeed, as mentioned above, the set $V_{1}\cup\cdots\cup V_{s-1}$ induces a
$T_{s-1}\left(  n-i\right)  ;$ hence,
\[
i\left(  n-i\right)  +t_{s-1}\left(  n-i\right)  =m,
\]
and so,%
\[
i\left(  n-i\right)  +\frac{s-1}{2s}\left(  n-i\right)  ^{2}=m-t_{s-1}\left(
n-i\right)  +\frac{s-1}{2s}\left(  n-i\right)  ^{2}.
\]
Set
\[
m^{\prime}=m-t_{s-1}\left(  n-i\right)  +\frac{s-1}{2s}\left(  n-i\right)
^{2}%
\]
and note that $i=q\left(  m^{\prime}\right)  .$ In view of Claim \ref{cl12},
we obtain%
\[
k_{r}\left(  H_{i}\right)  \leq\binom{s-1}{r-1}\left(  \frac{n-i}{s-1}\right)
^{r-1}i+\binom{s-1}{r}\left(  \frac{n-i}{s-1}\right)  ^{r}=f\left(  m^{\prime
}\right)  ,
\]
completing the proof.
\end{proof}

\begin{claim}
\label{cl10}$f^{^{\prime}}\left(  x\right)  =\binom{s-2}{r-2}p^{r-2}.$
\end{claim}

\begin{proof}
From (\ref{feq}) we have%
\[
f\left(  x\right)  =\binom{s-1}{r-1}\left(  \frac{s-r}{r}p^{r}+p^{r-1}%
q\right)  ,
\]
and so,
\[
f^{\prime}\left(  x\right)  =\binom{s-1}{r-1}\left(  \left(  s-r\right)
p^{r-1}p^{\prime}+\left(  r-1\right)  p^{r-2}qp^{\prime}+p^{r-1}q^{\prime
}\right)  .
\]
From (\ref{c2}) and (\ref{c3}) we have%
\[
\left(  s-1\right)  p^{\prime}+q^{\prime}=0
\]
and
\[
\left(  s-1\right)  \left(  \left(  s-2\right)  pp^{\prime}+p^{\prime
}q+pq^{\prime}\right)  =\left(  s-1\right)  p^{\prime}\left(  q-p\right)
=x^{\prime}=1.
\]
Now the claim follows after simple algebra.
\end{proof}

We immediately see that $f\left(  x\right)  $ is increasing. Also, since
$p\left(  x\right)  $ is decreasing, $f^{^{\prime}}\left(  x\right)  $ is
decreasing too, implying that $f\left(  x\right)  $ is concave. This, in turn,
implies that $\widehat{f}\left(  x\right)  $ is concave.

For every $i=1,\ldots,\left\lfloor n/s\right\rfloor ,$ by Claim \ref{cl13}, we
have%
\[
k_{r}\left(  H_{i}\right)  \leq f\left(  m^{\prime}\right)  \leq\widehat
{f}\left(  m\right)  ,
\]
and since, by Claim \ref{cl9}, $k_{r}\left(  H\left(  n,m\right)  \right)  $
is linear for $m\in\left[  e\left(  H_{i}\right)  ,e\left(  H_{i+1}\right)
\right]  ,$ inequality (\ref{meq1}) follows.

To finish the proof of (\ref{meq}), note that by Taylor's formula, in view of
the concavity of $f\left(  x\right)  ,$ we have%
\begin{align*}
\widehat{f}\left(  m\right)   &  \leq f\left(  m+\frac{s-1}{8}\right)  \leq
f\left(  m\right)  +\frac{s-1}{8}f^{\prime}\left(  m\right)  =f\left(
m\right)  +\frac{s-1}{8}\binom{s-2}{r-2}p^{r-2}\\
&  \leq f\left(  m\right)  +\frac{s-1}{8}\binom{s-2}{r-2}\left(  \frac{n}%
{s-1}\right)  ^{r-2}<f\left(  m\right)  +sn^{r-2}\leq f\left(  m\right)
+\frac{n^{r}}{n^{2}-2m},
\end{align*}
completing the proof of Theorem \ref{thmeq}.$\hfill\square$\bigskip

\textbf{Acknowledgement }Thanks to Cecil Rousseau for helpful discussions and
to Alex Razborov for pointing out some mistakes in the initial version of the manuscript.

\end{document}